\documentclass[twoside,12pt]{article}

\usepackage{amsmath}
\usepackage{amssymb}
\usepackage{pxfonts}
\usepackage{graphicx}
\usepackage{eucal}
\usepackage{mathrsfs}
\usepackage{theorem}
\usepackage{pifont}
\usepackage{sectsty}
\usepackage{amscd}
\usepackage{color}
\usepackage{fancyhdr}
\usepackage{framed}
\usepackage[all]{xy}
\usepackage{hyperref}
\usepackage{enumerate}
\usepackage{tikz}
\usepackage{tkz-euclide}

\usetkzobj{all}

\definecolor{shadecolor}{rgb}{0.8,0.8,0.8}

\theoremheaderfont{\fontfamily{pzc}\bfseries\large}
\newtheorem{theorem}{Theorem}[section]

\newtheorem{lemma}[theorem]{Lemma}
\newtheorem{proposition}[theorem]{Proposition}

\newtheorem{definition}[theorem]{Definition}

{\theorembodyfont{\rmfamily}

\newcommand{\specexercise}[1]{}

\newenvironment{proof}{{\flushleft \emph{Proof}:}}{\hfill\ding{110}}

\newenvironment{examples}{{\flushleft {\fontfamily{pzc}\bfseries\large Examples}:}}{\hfill $\blacktriangle\blacktriangle\blacktriangle$}

\newcommand{\brk}[1]{\left(#1\right)}          

\newcommand{\mymat}[1]{\begin{pmatrix} #1 \end{pmatrix}}
\newcommand{\Norm}[1]{\left\| #1 \right\|}     

\newcommand{\secref}[1]{Section~\ref{#1}}
\newcommand{\figref}[1]{Figure~\ref{#1}}
\newcommand{\thmref}[1]{Theorem~\ref{#1}}
\newcommand{\defref}[1]{Definition~\ref{#1}}

\newcommand{\propref}[1]{Proposition~\ref{#1}}
\newcommand{\lemref}[1]{Lemma~\ref{#1}}

\newcommand{\beq}{\begin{equation}}
\newcommand{\eeq}{\end{equation}}

\newcommand{\Emph}[1]{{\slshape\bfseries #1}}  

\newcommand{\g}{\mathfrak{g}}
\newcommand{\h}{\mathfrak{h}}
\newcommand{\euc}{\mathfrak{e}}
\newcommand{\Volume}{d\text{Vol}_\g}
\newcommand{\Volumen}{d\text{Vol}_{\g_n}}
\newcommand{\VolumeFn}{d\text{Vol}_{F_n^\star \g_n}}

\newcommand{\Volumeh}{d\text{Vol}_{\mathfrak{h}}}
\newcommand{\textVol}{\text{Vol}}

\newcommand{\M}{{\mathcal{M}}}
\newcommand{\tM}{\tilde{\M}}
\newcommand{\N}{\mathcal{N}}
\newcommand{\R}{\mathbb{R}}
\newcommand{\frakM}{{\mathfrak{M}}}

\newcommand{\limn}{\lim_{n\to\infty}}

\newcommand{\tr}{\operatorname{tr}}
\newcommand{\dist}{\operatorname{dist}}

\newcommand{\SO}[1]{\text{SO}(#1)}

\newcommand{\e}{\varepsilon}

\newcommand{\il}{\langle}
\newcommand{\ir}{\rangle}

\newcommand{\len}{\operatorname{Len}}
\newcommand{\dis}{\operatorname{Dis}}
\newcommand{\DIS}{\operatorname{\mathfrak{Dis}}}

\newcommand{\incl}{\hookrightarrow}
\DeclareMathOperator*{\limarrow}{\longrightarrow}

\numberwithin{equation}{section}

\begin{document}

\title{Limits of elastic models of converging Riemannian manifolds}
\author{Raz Kupferman and Cy Maor}
\date{}
\maketitle

\begin{abstract}
In non-linear incompatible elasticity, the configurations are maps from a non-Euclidean body manifold into the ambient Euclidean space, $\R^k$.
We prove the $\Gamma$-convergence of elastic energies for configurations of a converging sequence, $\M_n\to\M$, of body manifolds. This convergence result has several implications: (i) It can be viewed as a general structural stability property of the elastic model. (ii) It applies to certain classes of bodies with defects, and in particular, to the limit of bodies with increasingly dense edge-dislocations. (iii) It applies to approximation of elastic bodies by piecewise-affine manifolds. In the context of continuously-distributed dislocations, it reveals that the torsion field, which has  been used traditionally to quantify the density of dislocations, is immaterial in the limiting elastic model.
\end{abstract}

\tableofcontents

\section{Introduction}

One of the central notions in geometric theories of continuum mechanics, is that of a \emph{body manifold}, $\M$, whose points represent material elements. Mathematically, a body manifold is a topological, or differentiable manifold. Different types of continuum systems are characterized by different geometric structures imposed on the body manifolds. Body manifolds of elastic solids are commonly smooth manifolds endowed with a Riemannian metric, i.e., Riemannian manifolds. A \emph{configuration} of a body is an embedding of the body manifold into the ambient $k$-dimensional Euclidean space. In hyper-elastic materials, both static and dynamics properties of the material are dictated by an elastic energy, which is an integral measure of local distortions of the configurations.

In classical elasticity, the body manifold is assumed to be Euclidean, implying that it can be identified with a subset $\Omega$ of Euclidean $k$-dimensional space. The natural inclusion $\iota:\Omega\incl\R^k$ is called a rest, or a \emph{reference configuration}, and it is a state of zero elastic energy. 
In the last several years, there has been a growing interest in bodies that are \emph{pre-stressed}. 
Pre-stressed bodies are modeled as Riemannian manifolds $(\M,\g)$, where the \emph{reference metric} $\g$ is non-flat, i.e., has a non-zero Riemann curvature tensor. Thus, it cannot be embedded isometrically into Euclidean space. In particular, there is no notion of reference configuration. 

The elastic theory of pre-stressed bodies is commonly known as the theory of \emph{non-Euclidean}, or \emph{incompatible} elasticity. In its simplest versions, assuming material isotropy, the elastic energy associated with a configuration $u:\M\to \R^k$ is a distance of that configuration from being an isometric embedding (see e.g.~\cite{ESK08,KS14,LP10}). A prototypical energy density is
$\dist^p(du, \SO{\g,\euc})$ for some $p>1$, where $\euc$ is the Euclidean metric in $\R^k$ and $\SO{\g,\euc}$ is the set of orientation and inner-product preserving maps $T\M\to \R^k$ \cite{KS14,LP10}. A precise definition of this distance is given in the next section.

The theory of incompatible elasticity has numerous applications. It was proposed originally in the 1950s in the context of crystalline defects (see e.g.~Kondo~\cite{Kon55} and Bilby and co-workers~\cite{BBS55,BS56}). Then, the non-Euclidean metric structure associated with the defects is singular. 
In recent years, incompatible elasticity is motivated by studies of growing tissues, thermal expansion, and other mechanisms involving differential expansion of shrinkage \cite{ESK08,AESK11,AAESK12,OY10,KES07}; 
in these systems the intrinsic geometry is typically smooth.

In the context of crystalline defects, an important field of interest concerns distributed defects. 
Models of continuously-distributed defects were developed during the 1950s and 1960s.
Body manifolds of bodies with distributed defects are endowed with structure additional to a metric. For example, bodies with continuously-distributed dislocations are modeled by \emph{Weitzenb\"ock manifolds} $(\M,\g,\nabla)$, where $\nabla$ is a flat connection consistent with $\g$ (that is, a metric connection), whose torsion tensor represents the distribution of the dislocations \cite{Nye53,BBS55,Kro81}; see also the more recent  literature \cite{MR02,YG12b}.

The modeling of a body with distributed dislocations by a Weitzenb\"ock manifold is phenomenological, rather than mechanical. 
In particular, it is not associated with a class of constitutive relations (or elastic energies), and it is not clear how does $\nabla$ manifest (if at all)  in the response to deformation and loading (as pointed out in Section 1a in \cite{Wan67}).
This is in contrast to bodies with finitely many dislocations, which can be modeled as (singular) Riemannian manifolds with no additional structure (no torsion field), and for which standard elastic energies are applicable \cite{MLAKS15}.

There exists a vast literature on the mechanics of bodies with dislocations. However, those typically either use different phenomenological models for describing the dislocations (see e.g. \cite{CGO15}), or assume general classes of elastic energies that may or may not relate to the connection (e.g. \cite{CK13}, which does not use Weitzenb\"ock manifolds explicitly, however their choice of crystalline structure is equivalent to a choice of a flat connection), or only rely on the Riemannian part when considering mechanical response (\cite{YG12b}).

In \cite{KM15} and \cite{KM15b} we showed that Weitzenb\"ock manifolds (with non-zero torsion) can be obtained as rigorous limits of (torsion-free) singular Riemannian manifolds. 
Thus, the phenomenological model of a body with continuously distributed dislocations is a limit of bodies with finitely-many singular dislocations, as the density of the dislocations tends to infinity.
This new notion of converging manifolds calls for a rigorous derivation of a mechanical model for bodies with continuously-distributed dislocations:
Assuming a mechanical model for bodies with finitely-many singular dislocations,  is there a limiting mechanical model for  the limiting Weitzenb\"ock manifolds? 

The main question addressed in this paper is the following: Given a sequence of converging manifolds endowed with elastic energies depending continuously (in a precise sense) on the metric structure, what can be said about the $\Gamma$-limit of these energies?
To minimize technicalities, we consider systems free of external forces or constraints (note that the non-Euclidean structure renders such systems non-trivial). Body forces and boundary conditions can be included, if needed, in a standard way (see the Discussion section).
 
Our main result (Theorem \ref{thm:Gamma_convergence_general}) can be summarized as follows: 
\begin{quote}
{\em
Let $(\M_n,\g_n)$ be a sequence of body manifolds, with corresponding elastic energy densities $W_{(\M_n,\g_n)}$ satisfying boundedness and coercivity conditions, and depending continuously on the metric $\g_n$ (see Section~\ref{sec:energy_of_families} for a precise definition). If $(\M_n,\g_n)\to (\M,\g)$ uniformly (see Definition~\ref{def:uniform_convergence}), then the elastic energies $\Gamma$-converge to the relaxation of an energy with density $W_{(\M,\g)}$; if $(\M_n,\g_n)\to (\M,\g)$ in a weaker sense (see Definition \ref{def:weak_convergence}), then the relaxation of $W_{(\M,\g)}$ is an upper-bound to every $\Gamma$-convergent subsequence.
}
\end{quote}

As mentioned above,  it is shown in \cite{KM15, KM15b} that any 2D Weitzenb\"ock manifold $(\M,\g,\nabla)$ can be obtained as a limit of bodies with finitely many dislocations $(\M_n,\g_n,\nabla_n)$, where $\nabla_n$ is the Levi-Civita connection. The convergence of the Riemannian part, $(\M_n,\g_n)\to(\M,\g)$, is with respect to the weaker notion of convergence. Yet, a slight modification of our construction yields uniform convergence.

For a Weitzenb\"ock manifold $(\M,\g,\nabla)$ to constitute an adequate elastic model for a body with distributed dislocations, one would expect to have an elastic energy $E_{(\M,\g,\nabla)}$ associated with it. 
Since $(\M,\g,\nabla)$ is an effective limit model of bodies with finitely many defects, $E_{(\M,\g,\nabla)}$ should be a limit of the energies associated with these bodies. 
In the case where $\nabla=\nabla^{LC}$ is the Levi-Civita connection, the body has no continuously-distributed dislocations, so it is natural to choose $E_{(\M,\g,\nabla^{LC})}=E_{(\M,\g)}$, where $E_{(\M,\g)}$ is a standard non-Euclidean elastic energy (say, with density $\dist^p(du, \SO{\g,\euc})$). 
Our analysis shows that in this case $E_{(\M,\g,\nabla)}$ (or more accurately, its relaxation) would be independent of $\nabla$ even if it is not the Levi-Civita connection (and thus contains torsion).

This paper is concerned with isotropic materials, in which the elastic energy is derived from the Riemannian metric of a body manifold (the reference metric), which is fixed. 
In other models, involving anisotropy or defect dynamics, the connection $\nabla$ (or equivalently its torsion field) can still play a role in a limit energy functional. 
This lies outside the scope of this paper, and it is a natural topic for further research.

In addition, our main theorem implies the \emph{structural stability} of non-Euclidean elasticity under certain perturbations of the reference metric, as well as the convergence of certain approximation methods, based on locally-Euclidean approximations of body manifolds. 
These applications are elaborated in the discussion (Section \ref{sec:discussion}).

The structure of the paper is as follows: In Section \ref{sec:settings}, we define notions of convergence for body  manifolds, and define an $L^p$-topology for functions defined on converging manifolds.
In Section \ref{sec:energy_of_families}, we define a class of elastic energy functionals for configurations of body manifolds,  and prove, in particular, that the energy densities $\dist^p(\cdot, \SO{\g,\euc})$ belong to this class. While the definitions in this section are straightforward, it is the first time (to the best of our knowledge) 
that a convergence analysis relies on a precise quantitative relation between the metric structure and the elastic energy density.
In Section \ref{sec:Gamma_conv} we state and prove the main $\Gamma$-convergence result, and in Section \ref{sec:discussion} we discuss applications and limitations of our results, as well as some open questions.

\section{Settings}
\label{sec:settings}

\subsection{Definitions and notations}

Let $(V,\g)$ and $(W,\h)$ be two oriented $k$-dimensional inner-product spaces.
For a linear map $A:V\to W$ we denote by $|A|_\infty$ the operator norm of $A$, that is 
\[
|A|_\infty = \sup_{0\ne v\in V} \frac{|A(v)|_\h}{|v|_\g},
\]
and by $|A|_2=\tr (A^TA)$ the inner-product (Frobenius) norm induced by $\g$ and $\h$ .
Note that
\beq
\label{eq:equiv_of_norms}
|A|_\infty \le |A|_2 \le k\, |A|_\infty.
\eeq
When the exact norm is irrelevant or clear from the context, we simply write $|A|$.

We denote by $\dist_{\g,\h}$ (resp. $\dist^\infty_{\g,\h}$) the distance function on $L(V,W)$ with respect to the inner-product (resp. operator) norm induced by $\g$ and $\h$. 
We extend it to subsets of $L(V,W)$ as a Hausdorff distance. 

We denote by $\SO{\g,\h}$ the set of inner-product and orientation-preserving isomorphisms $(V,\g) \to (W,\h)$.
The \emph{distortion} of a map $A\in L(V,W)$ is defined as 
\beq
\label{eq:def_local_dis}
\DIS A = \dist_{\g,\h}(A,\SO{\g,\h}).
\eeq

All the above is extended to vector bundles equipped with inner-products in the standard way.
If $A$ is orientation preserving, and $\sigma_1,\ldots,\sigma_k$ are the singular values of $A$, then $\DIS A = \sqrt{(\sigma_1-1)^2+\ldots +(\sigma_k-1)^2}$.

Throughout the paper, we consider derivatives of maps $F:\M\to \N$ in the following way: Pointwise, for every $p\in \M$, we consider $(dF)_p : T_p\M \to T_{F(p)}\N$ as a map between vector spaces.  Globally, $dF$ is considered as a map $T\M \to F^*T\N$, where $F^*T\N$ is a vector bundle over $\M$, with the fiber $(F^*T\N)_p$ identified with the fiber $T_{F(p)}\N$. This way $dF$ is a bundle map over $\M$, thus separating its linear part  from its nonlinear part (the projection of $dF$ on the base space). 
Likewise, we denote by $F^*$ the pullback of tensor fields (such as Riemannian metrics), considered as sections of tensor products of $T\N$ and $T^*\N$. 
This should not be confused with the closely related pullback involving composition with $dF$, which we denote by $F^\star$. For example, if $\h$ is a Riemannian metric on $\N$, then $F^*\h$ is an inner product on the vector bundle $F^*T\N$, whereas $F^\star\h$ is an inner product on $T\M$ (hence a Riemannian metric on $\M$, unlike $F^*\h$), which is defined by
\[
F^\star\h(v,w) = F^*\h(dF(v),dF(w)),
\]
for every two vector fields $v,w\in \Gamma(T\M)$, whereas for every $p\in \M$ we have, 
\[
F^*\h_p(dF_p(v_p),d_pF(w_p)) = \h_{F(p)}(dF_p(v_p),d_pF(w_p)).
\]

\subsection{Body manifolds and their morphisms}

Body manifolds are a general notion  in mechanics, whose precise definition depends on the specific context.
In this section we define the class of manifolds to which our results refer. 
Since we are interested in bodies with defects, 
our concept of body manifold allows for singularities, which implies that we cannot require a smooth structure on the entire manifold.

\begin{definition}
A \Emph{body manifold} is a quadruple $(\M,d,\tM,\g)$, where $\M$ is a $k$-dimensional compact, oriented, connected topological manifold with corners and $d$ is a distance function on $\M$.  $\tM\subset\M$ is an open smooth submanifold, such that $\M\setminus\tM$ has a $k$-dimensional Hausdorff measure zero with respect to $d$. $\g$ is a Riemannian metric on $\tM$, consistent with the distance $d$ in the following sense: for every $p,q\in \M$, $d(p,q)$ is the infimum over the lengths 
\[
\len(\gamma) = \int_I \sqrt{\g(\dot{\gamma}(t), \dot{\gamma}(t))}\, dt.
\] 
of continuous paths $\gamma:I\to\M$ that are a.e. smooth. In particular,  $\gamma(t)\in\tM$ for all $t$ except perhaps for a set of measure zero.
\end{definition}

The consistency between $\g$ and $d$ ensures that there are no ``shortcuts" through the non-smooth parts of the body, i.e. that the Riemannian metric induces the distance.
Note also that the Riemannian metric induces a measure on $\tM$---the volume form. This measure can be extended into a measure $\mu$ on $\M$ by setting $\mu(\M\setminus\tM)=0$. Since $d$ and $\g$ are consistent,  the null sets of $\mu$ coincide with the null sets of the $k$-dimensional Hausdorff measure.

\begin{examples}\
\begin{enumerate}
\item
The trivial example: Every compact, oriented, connected Riemannian manifold with corners $(\M,\g)$ is a body manifold with $\tM = \M$ and $d$ induced by the Riemannian metric.

\item
A cone is a body manifold: it is a two-dimensional topological manifold hemeomorphic to a disc, endowed with a locally Euclidean metric everywhere but at one point---the tip of the cone. In the mechanical context, a cone is a  disclination-type defect.

\item
Every piecewise-affine manifold is a body manifold. The smooth component $\tM$ may be disconnected. Piecewise-affine manifolds are prevalent in mechanics in the context of numerical approximations. 
\end{enumerate}
\end{examples}

We now define morphisms between body manifolds: 
these are bi-Lipschitz homeomorphisms that are local diffeomorphisms whenever the differential is defined (the smooth parts need not be diffeomorphic).

\begin{definition}
Let $(\M,d_\M,\tM,\g_\M)$ and $(\N,d_\N,\tilde{\N},\g_\N)$ be body manifolds. A \textbf{morphism} between those manifolds is a bi-Lipschitz homeomorphism $F:\M\to\N$, such that the restriction of $F$ to $\tM\cap F^{-1}(\tilde{\N})$ (which is a set of full measure, since $F^{-1}$ is Lipschitz) is a smooth embedding.
\end{definition}

\begin{examples}\
\begin{enumerate}
\item
Every diffeomorphism between Riemannian manifolds is a body manifold morphism.

\item
A cone can be parametrized by polar coordinates, $(r,\theta)$, with a metric whose components 
\[
\g(r,\theta) = \mymat{1 & 0 \\ 0 & \alpha^2 r^2},
\qquad 0 < \alpha \ne1,
\] 
are defined for every $r>0$. The identity map into a Euclidean disc is a body manifold morphism. Note that the smooth parts of the cone and the disc are not diffeomorphic.

\item
Maps from smooth Riemannian manifolds to piecewise-affine approximations are body manifold morphisms.
\end{enumerate}
\end{examples}

Elasticity is concerned with material response to distortions. In our context, where a body has a two metric structure---a distance function and a Riemannian metric---we distinguish between local and global distortions of body manifold morphisms:

\begin{definition}
\label{def:distortions}
Let $(\M,d_\M,\tM,\g_\M)$ and $(\N,d_\N,\tilde{\N},\g_\N)$ be body manifolds and let $F:\M\to\N$ be a morphism. The \Emph{local distortion} of $F$ is the distortion of the linear map $dF$ as defined in \eqref{eq:def_local_dis}, i.e., it is the map $\DIS dF: \tM\cap F^{-1}(\tilde{\N})\to[0,\infty)$,
\[
\DIS dF = \dist_{\g_\M,F^*\g_\N}(dF,\SO{\g_\M,F^*\g_\N}).
\]
The \Emph{global distortion} of $F$ is a non-negative number defined as 
\[
\dis F = \sup_{p,q\in\M} |d_\M(p,q) - d_\N(F(p),F(q))|.
\]
\end{definition}

\subsection{Convergence of body manifolds}
\label{sec:conv_of_manifolds}

In this section we define two modes of convergence for body manifolds, which, loosely speaking, correspond to uniform and non-uniform convergence of the Riemannian metrics.

\begin{definition}[Uniform convergence of body manifolds]
\label{def:uniform_convergence}
Let $(\M_n,d_n,\tM_n,\g_n)$ be a sequence of body manifolds and let $(\M,d,,\tM,\g)$ be a body manifold.
We say that the sequence $\M_n$ \textbf{converges uniformly} to $\M$, if there exists a sequence of  body manifold morphisms $F_n:\M \to \M_n$, such that the local distortion vanishes uniformly,
\beq
\label{eq:uniform_asymptotic_rigidity}
\limn \|\DIS dF_n\|_\infty = 0.
\eeq
\end{definition}

\begin{definition}[Mean convergence of body manifolds]
\label{def:weak_convergence}
Let $(\M_n,d_n,\tM_n,\g_n)$ be a sequence of body manifolds and let $(\M,d,\tM,\g)$ be a body manifold.
We say that the sequence $\M_n$ \textbf{converges in the mean} to $\M$, if there exists a sequence of  body manifold morphisms $F_n:\M \to \M_n$,  such that
\begin{enumerate}
\item $F_n$ are uniformly bi-Lipschitz, i.e. there exists a constant $C>0$, independent of $n$, such that
\beq
\label{eq:uni_bi_Lipschitz}
|(d F_n)_p| , |((dF_n)_p)^{-1}| < C,
\eeq
for every $p\in \M$ where $dF_n$ is defined.
(Note that $(dF_n)^{-1} = F_n^*(d(F_n^{-1}))$.)
\item $F_n$ are approximate distance-preserving as maps between metric spaces: the global distortion vanishes asymptotically,
\beq
\label{eq:vanishing_dis}
\limn\dis F_n = 0.
\eeq
\item $F_n$ are asymptotically rigid in the mean:
\beq
\label{eq:asymptotic_rigdity}
\limn \int_\M \DIS dF_n \,\Volume = 0.
\eeq

\item The volume forms converge uniformly:
\beq
\label{eq:volume_conv}
\frac{d\text{Vol}_{F_n^\star \g_n}}{\Volume} \to 1 \qquad \text{in} \,\, L^\infty.
\eeq
\end{enumerate}
\end{definition}

To simplify notations, we will denote the body manifolds $(\M_n,d_n,\tM_n,\g_n)$ and $(\M,d,\tM,\g)$
 by $\M_n$ and $\M$, whenever no confusion should arise.
 
These definitions, and especially the definition of mean convergence, may seem a bit convoluted, so we first provide the rationale behind them.
As our main motivation for this work is the convergence of bodies with dislocations, we consider notions of convergence that (i) are satisfied by converging bodies with dislocations considered in \cite{KM15,KM15b} (further details are given in the examples section below); and (ii) are strong enough to imply the $\Gamma$-convergence of associated elastic energies.

The crux in each type of convergence is the way $\DIS F_n$ converges to zero. When the convergence is in $L^\infty$ (uniform convergence), it follows automatically that $F_n$ are uniformly bi-Lipschitz and that the volume forms converge uniformly; these properties are needed for our $\Gamma$-convergence proof. 
When $\DIS F_n\to 0$ only in $L^1$ (mean convergence) both the uniform by-Lipschitz property and volume convergence are not guaranteed, hence have to be imposed explicitly, as Conditions \eqref{eq:uni_bi_Lipschitz} and \eqref{eq:volume_conv} (which are satisfied by our main examples, see below).
Future improvements of the $\Gamma$-convergence proof may allow to relax these conditions.

Condition \eqref{eq:vanishing_dis} is of ``global" nature, and unlike the other conditions, does not involve the differentials $dF_n$ explicitly. Furthermore,
it plays no explicit role in the $\Gamma$-convergence proof; its  role is to ``enforce" Gromov-Hausdorff convergence (see below), and as a result, the uniqueness of the limit (a limit body independent of the mappings $F_n$).
It is possible that the other conditions in Definition \ref{def:weak_convergence} suffice for a unique limit, in which case Condition \eqref{eq:vanishing_dis} can be omitted. This is, however, a pure question of geometric rigidity, and it is beyond the scope of this paper. It is further discussed in the open questions part of Section \ref{sec:discussion}.

In the rest of this subsection we prove some properties of convergent sequences, and give some examples.

\begin{lemma}\label{lm:every_p_rigidity}
If $\M_n$ converges to $\M$ in the mean, and $F_n:\M\to\M_n$ are maps that realize the convergence, then for every $p\in[1,\infty)$
\beq
\label{eq:asymptotic_rigdity_p}
\limn \int_\M (\DIS dF_n)^p \,\Volume = 0,
\eeq
and
\beq
\label{eq:reverse_asymptotic_rigdity}
\limn \int_{\M_n}(\DIS dF^{-1}_n)^p \,\Volumen = 0.
\eeq
\end{lemma}

\begin{proof}
Since 
\[
\DIS dF_n \le |dF_n|_2 + k, 
\]
and since $|dF_n|_2$ is uniformly bounded by  \eqref{eq:uni_bi_Lipschitz}, it follows from the Bounded Convergence Theorem that $L^1$-convergence \eqref{eq:asymptotic_rigdity} implies $L^p$-convergence 
\eqref{eq:asymptotic_rigdity_p}.

Similarly, it is enough to prove \eqref{eq:reverse_asymptotic_rigdity} for $p=1$.
It follows from \eqref{eq:asymptotic_rigdity} that for every $\e>0$ there exist sets $A_n\subset \M$ 
whose complements have asymptotically vanishing volume, $\textVol(\M\setminus A_n) \to 0$, 
in which $\DIS dF_n < \e$. 
It follows that the singular values of $dF_n$ with respect to the frame are in the interval $(1-\e,1+\e)$, hence all the singular values of $(dF_n)^{-1}$ are in the interval $((1+\e)^{-1},(1-\e)^{-1})\subset (1-2\e,1+2\e)$, from which follows that for every point in $A_n$,
\beq
\label{eq:rigidity_in_A_n}
F_n^* \DIS dF_n^{-1}  < 2\e\cdot \sqrt{k}.
\eeq
From \eqref{eq:rigidity_in_A_n} and the uniform bound \eqref{eq:uni_bi_Lipschitz}, it follows that
\[
\begin{split}
\int_\M F_n^* \DIS dF_n^{-1} \, \Volume &\le
2\sqrt{k}\e \textVol(A_n) + C\textVol(\M\setminus A_n) \\
&\le 2\sqrt{k}\textVol(\M) \e + o(1) 
\end{split}
\qquad\text{as $n\to \infty$.}
\]
Since $\e$ is arbitrary,
\[
\limn \int_\M  F_n^* \DIS dF_n^{-1}
\, \Volume = 0.
\]
Using the uniform convergence of the volume \eqref{eq:volume_conv},
\[
\limn \int_\M  F_n^* \DIS dF_n^{-1}
\, \VolumeFn = 0,
\]
from which \eqref{eq:reverse_asymptotic_rigdity}  for $p=1$ follows by a change of variables.
\end{proof}

\subsubsection{Relations to other modes of convergence}
\label{sec:relations_between_convergence_notions}

\begin{enumerate}
\item Uniform convergence is stronger than mean convergence. Indeed, \eqref{eq:uniform_asymptotic_rigidity} implies \eqref{eq:uni_bi_Lipschitz} and \eqref{eq:asymptotic_rigdity}. 
Uniform convergence of volumes \eqref{eq:volume_conv} follows from the inequality
\[
\left|\frac{\VolumeFn}{\Volume}-1\right| \le (\DIS dF_n + 1)^k - 1,
\]
see Lemma 4.5 in \cite{KM15} for details.
Finally, let $\gamma$ be a curve whose lengths $\ell_\g(\gamma)$ and $\ell_{F_n^\star\g_n}(\gamma)$ with respect to $\g$ and $F_n^\star\g_n$ are well-defined. 
The uniform convergence \eqref{eq:uniform_asymptotic_rigidity} implies that $|\ell_\g(\gamma)-\ell_{F_n^\star\g_n}(\gamma)|\to 0$ over all such curves. 
Moreover, for every $R>0$, this convergence is uniform over all curves of length less or equal $R$. 
This implies uniform convergence of the distances $d_{F_n^\star\g_n}\to d_{\g}$ (the distance functions induced on $\M$ by the Riemannian metrics $F_n^\star\g_n$ and $\g$). Since $d_{\g} = d$ and $d_{F_n^\star\g_n}$ is the pullback by $F_n$ of $d_n$, this implies the asymptotic vanishing \eqref{eq:vanishing_dis} of the global distortion.

\item Both types of convergence are weaker than $(m,\alpha)$-H\"older convergence of smooth manifolds, for any $m\ge 0$ and $\alpha\in (0,1)$.
Indeed, by definition, $(\M_n,\g_n)\to (\M,\g)$ in the $C^{m,\alpha}$-topology if there exists diffeomorphisms $F_n:\M\to \M_n$ such that $F_n^\star \g_n\to \g$ in the $C^{m,\alpha}$-topology, i.e., the components of the metric converge in the $C^{m,\alpha}$-topology in any coordinate chart (see \cite[Chapter 10]{Pet06} for details). 
This implies \eqref{eq:uniform_asymptotic_rigidity}, hence uniform convergence.
	
Thus, all the results presented in this paper apply \emph{a fortiori} to H\"older-converging manifolds.
	
\item Mean convergence of $\M_n$ to $\M$ implies measured Gromov-Hausdorff convergence of the \emph{measured metric spaces} $(\M_n,d_{\g_n},\textVol_{\g_n})$ to $(\M,d_\g,\textVol_{\g})$ (see \cite[Chapter 10]{Pet06} for details). 
Indeed, \eqref{eq:vanishing_dis} implies Gromov-Hausdorff convergence, and \eqref{eq:volume_conv} implies weak convergence of the measures $\textVol_{F_n^\star\g_n}$ to $\textVol_{\g}$.
\end{enumerate}

\subsubsection{Examples}
\label{sec:Remarks_and_Examples}

\begin{enumerate}
\item The convergence defined in \cite{KM15,KM15b} in the context of distributed edge-dislocations is weaker than mean convergence, but on the other hand, it also embodies the convergence of an additional structure---a flat metric connection. In the terminology of the present paper, \cite{KM15,KM15b} deal with the convergence of quintuples $(\M,d,\tM,\g,\nabla)$, where $\nabla$ is a flat metric connection on $\tM$.
The explicit sequences of manifolds constructed in \cite{KM15,KM15b} exhibit mean convergence (see Propositions~1 and 3 in \cite{KM15}).
Therefore, the main theorem in \cite{KM15b} implies that generic smooth, $2$-dimensional surfaces can be obtained as mean convergence limits of locally-Euclidean surfaces with distributed edge-dislocations.

\item The constructions in \cite{KM15,KM15b} are composed from building blocks, each containing a pair of disclinations of opposite charge, as illustrated in \figref{fig:dislocation}.
In the $n$-th stage, the two disclinations in each building block have angles $\pm2\theta$, i.e., independent of $n$ and identical in all blocks, whereas the distance $d$ between the disclinations is of order $1/n^2$. This construction yields in each block a dislocation of magnitude $2d\,\sin\theta \sim 1/n^2$.

\begin{figure}
\begin{center}
\begin{tikzpicture}[scale=1.2]
	\tkzDefPoint(0,4){A}
	\tkzDefPoint(0,2.5){A1}
	\tkzDefPoint(2,2.5){pm}
	\tkzDefPoint(4,2.5){dummy}
	\tkzDefPoint(3,2.25){mid}
	\tkzDefPoint(4,2){pp}
	\tkzDefPoint(6,2){D1}
	\tkzDefPoint(6,4){D}
	\tkzDrawPoints(A,A1,pm,pp,D1,D);
	\tkzLabelPoint[left](A1){$x$};
	\tkzLabelPoint[right](D1){$y$};
	\tkzLabelPoint[above](pm){$p_-$};
	\tkzLabelPoint[above](pp){$p_+$};
	\tkzLabelPoint[below](mid){$d$};
	\tkzDrawSegment(A,D);
	\tkzDrawSegment(D,D1);
	\tkzDrawSegment(A,A1);
	\tkzDrawSegment(A1,pm);
	\tkzDrawSegment(pp,pm);
	\tkzDrawSegment(pp,D1);
	\tkzDrawSegment[dashed](pm,dummy);
	\tkzLabelAngle(dummy,pm,pp){{\small $\theta$}};
	\tkzMarkAngle[size=1.2](pp,pm,dummy);
	\tkzMarkSegment[mark=s|](A1,pm)
	\tkzMarkSegment[pos=0.75, mark=|||](pm,pp)
	\tkzMarkSegment[mark=s||](pp,D1)

	\tkzDefPoint(0,-1){B}
	\tkzDefPoint(0,0.5){B1}
	\tkzDefPoint(2,0.5){pm}
	\tkzDefPoint(4,0.5){dummy}
	\tkzDefPoint(3,0.75){mid}
	\tkzDefPoint(4,1){pp}
	\tkzDefPoint(6,1){C1}
	\tkzDefPoint(6,-1){C}
	\tkzDrawPoints(B,B1,pm,pp,C1,C);
	\tkzLabelPoint[left](B1){$x$};
	\tkzLabelPoint[right](C1){$y$};
	\tkzLabelPoint[above](pm){$p_-$};
	\tkzLabelPoint[above](pp){$p_+$};
	\tkzLabelPoint[above](mid){$d$};
	\tkzDrawSegment(B,C);
	\tkzDrawSegment(C,C1);
	\tkzDrawSegment(B,B1);
	\tkzDrawSegment(B1,pm);
	\tkzDrawSegment(pp,pm);
	\tkzDrawSegment(pp,C1);
	\tkzDrawSegment[dashed](pm,dummy);
	\tkzLabelAngle(dummy,pm,pp){{\small $\theta$}};
	\tkzMarkAngle[size=1.2](dummy,pm,pp);
	\tkzMarkSegment[mark=s|](B1,pm)
	\tkzMarkSegment[pos=0.75, mark=|||](pm,pp)
	\tkzMarkSegment[mark=s||](pp,C1)
\end{tikzpicture}
\end{center}
\caption{A single edge-dislocation realized as a dipole of disclinations at $p_-$ and $p_+$, by gluing the segments $[x,p_-]$, $[p_-,p_+]$ and $[p_+,y]$ in the upper polygon to the matching segments in the lower polygon. The disclination angle is $2\theta$ and the distance between the dislocations is $|[p_-,p_+]|=d$, yielding a dislocation magnitude (identified with the size of the Burgers vector) $2d\sin(\theta)$.}
\label{fig:dislocation}
\end{figure}
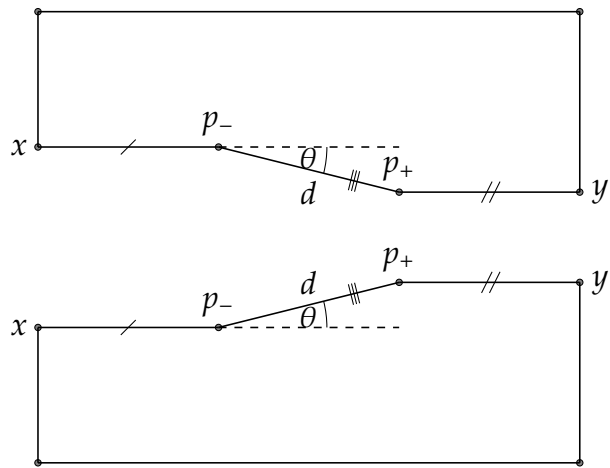

These constructions yield sequences of body manifolds that do not converge uniformly, but do converge in the mean. The lack of uniform convergence stems from the fact that the disclination angles  do not vanish as $n$ tends to infinity. When mapping the manifolds $\M_n$ into the limit manifold $\M$, one has to map curves such as $xp_-p_+y$ in \figref{fig:dislocation} to smooth curves. This always results in asymptotically small areas where $dF_n$ is bounded away from being a rigid transformation.

A slight modification of the constructions in \cite{KM15,KM15b} yields a sequence of locally-Euclidean surfaces with distributed edge-dislocations that converges {\em uniformly} to a smooth two-dimensional surface. For that, one has to take the angle $\theta$ in each building block to be of order $1/n^\e$ for some small $\e$, and set the distance $d$ between the disclinations such that the dislocation magnitude is the same as in the original construction (hence $d$ is of order $1/n^{2-\e}$). This construction yields the same limit as the original construction.

While it can be argued that vanishing disclination angles are ``less physical" than fixed ones (especially in the context of crystalline solids), this shows that any smooth surface with a continuous distribution of dislocations $(\M,d,\M,\g,\nabla)$ (since $\M$ is smooth $\tM=\M$) can be approximated {\em uniformly} by surfaces with finitely many dislocations  $(\M_n,d_n,\tM_n,\g_n,\nabla_n)$, where $\nabla_n$ is the Levi-Civita connection.

\item Another example of uniform convergence is the convergence of approximations of a surface via Euclidean triangulations: 
Any given surface can be triangulated by geodesic triangles whose edge-lengths are of order $1/n$ and whose angles are bounded away from $0$ and $\pi$. 
For $n$ large enough, each such triangle can be replaced by a Euclidean triangle of the same edge lengths. This yields a surface having disclination-type singularities at the vertices, while being locally Euclidean everywhere else.
As $n$ tends to infinity, these singular, locally Euclidean surfaces converge uniformly to the original surface.
Higher dimensional analogues to this construction are also possible.

\item
As an example of a sequence of manifolds converging to a smooth manifold in a weak sense, but not in the mean (and therefore neither uniformly), one can take any sequence of Riemannian manifolds $(\M_n,\g_n)$ that converges to $(\M,\g)$ while $\lim_n \textVol(\M_n) \ne \textVol(\M)$;
there are many such examples in the literature (see e.g. \cite{Iva98}). 

An example relevant to the homogenization of defects is the convergence of bodies with increasingly dense point-defects, as in \cite{KMR15}. There, $\lim_n \textVol(\M_n) > \textVol(\M)$.
In this example the maps $F_n:\M\to \M_n$ are far from being rigid, as $\DIS dF_n$ is uniformly bounded away from zero almost everywhere. 
Thus, the homogenization of point-defects does not fall under the framework of this paper.
\end{enumerate}

\subsection{Convergence of maps on converging manifolds}
\label{sec:L_p_top}

Having two notions of convergence for body manifolds, we proceed to define a topology for maps $f_n:\M_n\to\R^k$. 

\begin{definition}\label{def:emb_conv}
Let $\M_n$ be a sequence of body manifolds converging to a body manifold $\M$ (either uniformly or in the mean), and let $F_n: \M\to \M_n$ be body manifold morphisms that realize the convergence.
We say that a sequence $u_n\in L^p(\M_n;\R^k)$ converges to $u\in L^p(\M;\R^k)$ in $L^p$ (relative to $F_n$) if
\[
\| u_n \circ F_n - u\|_{L^p(\M;\R^k)} \to 0.
\]
\end{definition}

Note that this convergence depends on the maps $F_n$, which means that we do not have a general notion of convergence of sequences in $L^p(\M_n;\R^k)$ to a limit in $L^p(\M;\R^k)$. This convergence induces a natural topology on the disjoint union $(\sqcup_n L^p(\M_n;\R^k) \sqcup L^p(\M;\R^k)$; see \cite{KS08} for details. In the terminology of \cite{KS08}, we defined an \emph{asymptotic relation} between $L^p(\M_n;\R^k)$ and $L^p(\M;\R^k)$, since the sequence $F_n$ also realizes measured Gromov-Hausdorff convergence, as stated in the third item in \secref{sec:relations_between_convergence_notions}.

The following lemma establishes standard properties of $L^p$-convergence, adapted to converging manifolds:

\begin{lemma}\
\label{lm:L_p_convergence}
\begin{enumerate}
\item If $u_n\to u$ in $L^p$, then $u_n$ is a bounded sequence in $L^p(\M_n;\R^k)$, namely, $\| u_n\|_{L^p(\M_n;\R^k)}$ is bounded.

\item
If $u_n$ is bounded in $W^{1,p}(\M_n;\R^k)$ (i.e. $\| u_n\|_{W^{1,p}(\M_n;\R^k)}$ is bounded), then $u_n\circ F_n$ is uniformly bounded in $W^{1,p}(\M;\R^k)$, and in particular admits a weakly $W^{1,p}$-convergent subsequence.
\end{enumerate}
\end{lemma}

\begin{proof}
It is enough to prove the lemma under the assumption that $\M_n\to \M$ in the mean. 
Let $u_n\to u$ in $L^p$. 
By the triangle inequality,
\[
\|u_n\circ F_n\|_{L^p(\M;\R^k)} \le \|u_n\circ F_n - u\|_{L^p(\M;\R^k)} + 
\|u\|_{L^p(\M;\R^k)}.
\]
The first term tends, by definition, to zero. Next,
\[
\begin{split}
\|u_n\|^p_{L^p(\M_n;\R^k)} &= \int_{\M_n} |u_n|^p\, \Volumen = \int_{\M} |u_n\circ F_n|^p \, \VolumeFn \\
		&= \int_{\M} |u_n\circ F_n|^p \, \frac{\VolumeFn}{\Volume} \Volume \\
		&\le \Norm{\frac{\VolumeFn}{\Volume}}_\infty \|u_n \circ F_n\|^p_{L^p(\M;\R^k)},
\end{split}
\]
hence
\[
\limsup_{n\to\infty} \|u_n\|_{L^p(\M_n;\R^k)} \le \|u\|_{L^p(\M;\R^k)},
\]
which proves the first part.

For the second part, assume that $u_n$ is bounded in $W^{1,p}(\M_n;\R^k)$. 
In particular, $u_n$ is bounded in $L^p(\M_n;\R^k)$.
The same calculation as above yields that $u_n\circ F_n$ is bounded in $L^p(\M;\R^k)$.
Moreover, 
\beq
\label{eq:W_1_p_weak_compactness}
\begin{split}
\|d(u_n\circ F_n)\|_p^p &= \int_M |d(u_n\circ F_n)|^p \, \Volume \le \int_M F_n^*|du_n|^p \cdot |dF_n|^p \, \Volume \\
&\le C \int_M F_n^*|du_n|^p \, \Volume = C \int_{\M_n} |du_n|^p \, d\text{Vol}_{(F_n)_\star \g} \\
&= C \int_{\M_n} |du_n|^p \, \frac{d\text{Vol}_{(F_n)_\star \g} }{\Volumen} \Volumen \\
&\le C\|du_n\|^p_{L^p(\M_n;\R^k)} \Norm{\frac{d\text{Vol}_{(F_n)_\star \g} }{\Volumen}}_\infty
\le C' \|du_n\|^p_{L^p(\M_n;\R^k)},
\end{split}
\eeq
where the norms of $du_n$, $dF_n$ and $d(u_n\circ F_n)$ at a point $p$ are in the space $L(T_p\M_n, \R^k)$, as described in Section \ref{sec:def_and_notation}.
Between the first and the second line we used the uniform Lipschitz continuity of $F_n$. In the passage to the last line we used the uniform Lipschitz continuity of $F_n^{-1}$, and the fact that if $G:(\N,\h)\to (\N',\h')$ is a smooth map between $k$-dimensional Riemannian manifolds ($G= F_n^{-1}$ in our case), then Hadamard's inequality (see Lemma 4.5 in \cite{KM15} for details) implies
\[
\frac{d\textVol_{G^\star\h'}}{\Volumeh} \le |dG|^d.
\]
Together with the boundedness of $u_n\circ F_n$ in $L^p(\M;\R^k)$, \eqref{eq:W_1_p_weak_compactness} implies that $u_n\circ F_n$ is bounded in $W^{1,p}(\M;\R^k)$, which completes the proof.
\end{proof}

\section{Energy functionals on families of manifolds}
\label{sec:energy_of_families}

\begin{definition}
\label{def:elastic_energy}
Let $\frakM$ be a class of body manifolds.
\begin{enumerate}
\item An \textbf{energy density} on $\frakM$ is a function
\[
W: \bigsqcup_{(\M,d,\tM,\g)\in\frakM} T^*\tM\otimes \R^k \to \R.
\]
We denote the restriction of $W$ to $(\M,d,\tM,\g)$ by $W_{(\M,\g)}$.
\item An energy density on $\frakM$ is called \textbf{$p$-regular} for $p\in(1,\infty)$, if the following holds:
\begin{enumerate}
\item Regularity: For every $(\M,d,\tM,\g)\in \frakM$, $W_{(\M,\g)}$ is a Carath\'eodory function; see Appendix A in \cite{KM14} for the definition of Carath\'eodory functions in Riemannian settings.

\item Uniform coercivity: There exist $\alpha,\beta>0$ such that for every $(\M,d,\tM,\g)\in \frakM$,
\beq
W_{(\M,\g)}(A) \ge \alpha |A|^p - \beta ,\qquad \forall A\in  T^*\tM\otimes \R^k.
\label{eq:coercivity}
\eeq

\item Uniform boundedness: There exists a $\gamma>0$ such that for every $(\M,d,\tM,\g)\in \frakM$,
\beq
W_{(\M,\g)}(A) \le \gamma (|A|^p + 1),\qquad \forall A\in  T^*\tM\otimes \R^k.
\label{eq:boundedness}
\eeq

\item Lipschitz continuity in the metric: 
There exists a $C>0$ such that for every $(\M,d,\tM,\g),(\N,d',\tilde{\N},\h)\in \frakM$, 
linear isomorphism $L: T\tM \to T\tilde{\N}$, and $A\in  T^*\tilde{\N}\otimes \R^k$
\beq
\label{eq:property_d_energy_density}
|W_{(\M,\g)}(A\circ L)^{1/p} - L^* W_{(\N,\h)}(A)^{1/p} | < C (1 + L^* |A|) \DIS L.		
\eeq
\end{enumerate}

\item Let $W$ be a $p$-regular energy density on $\frakM$. Its \textbf{associated energy functional} is the function 
\[
E : \bigsqcup_{(\M,d,\tM,\g)\in \frakM} L^p(\M;\R^k) \to \R\cup \{+\infty\}
\]
defined by
\[
E_{(\M,\g)}[u] = 
\begin{cases}
\int_\M W(du) \,\Volume & u\in W^{1,p}(\M;\R^k) \\
	+\infty & u\in L^p(\M;\R^k) \setminus W^{1,p}(\M;\R^k).
\end{cases}
\]
\end{enumerate}
\end{definition}

Energy densities are normally defined for a single Riemannian manifold. Here, we define an energy density for a class of Riemannian manifolds. The crux of the matter is that the energy density for a given manifold at a given point only ``sees" the metric at that point.  Conditions $2(a)$--$(c)$ are standard regularity conditions. Condition $2(d)$ is the one that involves dependence on the metric. In particular, when reduced to a single  manifold, it implies homogeneity and isotropy.  Indeed, let $(\M,\g)$ be a Riemannian manifold, $x,y\in\M$,  $A\in T^*_x\M\times\R^k$ and  $B\in T^*_y\M\times\R^k$. If $A$ and $B$ are related by an isometry, namely, $A = B\circ L$ for some $L\in\SO{\g_x,\g_y}$, then
\[
W_{(\M,\g)}(A) = W_{(\M,\g)}(B).
\]

The motivation for the Lipschitz continuity \eqref{eq:property_d_energy_density} is that it is a key property satisfied by the prototypical energy density $\DIS du$, as proved in the next proposition:

\begin{proposition}
Let $(\M,d,\tM,\g)$ be a $k$-dimensional body manifold.
For every $p\in(1,\infty)$, the energy density
\beq
\label{eq:W_dist_p}
W_{(\M,\g)}(du) =  (\DIS du)^p,
\eeq
is $p$-regular, where $u$ is considered as a map $(\M,d,\tM,\g)\to(\R^k,\euc)$.
\end{proposition}

\begin{proof}
The regularity property (a) holds since $W_{(\M,\g)}$ is continuous on every fiber and we have the smoothness of the manifold on its base.

The coercivity (b) and boundedness (c) are immediate, hence it remains to prove (d).
Let $(\M,d,\tM,\g)$ and $(\N,d',\tilde{\N},\h)$ be body manifolds. Let $L: T\tM \to T\tilde{\N}$ be a linear isomorphism, and let $A\in  T^*\tilde{\N}\otimes \R^k$. We need to prove \eqref{eq:property_d_energy_density}.

The energy density \eqref{eq:W_dist_p} depends on $\g$ in two ways: via the metric with respect to which the distortion is measured, and via the set of local isometries $\SO{\g,\euc}$ whose distance from is being measured. 
We treat each dependence separately:
\[
\begin{split}
|W_{(\M,\g)}(A\circ L)^{1/p} - L^* W_{(\N,\h)}(A)^{1/p} | &= 
|\DIS (A\circ L)  - L^* \DIS A| \\
&\hspace{-2cm}=
|\dist_{\g,\euc}(A\circ L , \SO{\g,\euc}) - L^* \dist_{\h,\euc}(A , \SO{\h,\euc})| \\
&\hspace{-2cm} \le |\dist_{\g,\euc}(A\circ L , \SO{\g,\euc}) - 
\dist_{\g,\euc}(A\circ L , \SO{L^\star\h,\euc})| \\
&\hspace{-2cm}\qquad + |\dist_{\g,\euc}(A\circ L , \SO{L^\star\h,\euc}) - 
L^* \dist_{\h,\euc}(A , \SO{\h,\euc})| \\
&\hspace{-2cm} \le \dist_{\g,\euc}(\SO{\g,\euc},\SO{L^\star\h,\euc}) \\
&\hspace{-2cm}\qquad+ |\dist_{\g,\euc}(A\circ L , \SO{L^\star\h,\euc}) - L^* \dist_{\h,\euc}(A , \SO{\h,\euc})|.
\end{split}
\]
In the passage to the last inequality we used the fact that in any metric space $(X,d)$, with $x\in X$ and $A,B\subset X$,
\[
|d(x,A) - d(x,B)| \le d(A,B),
\]
where the distance on the right-hand side is the hausdorff distance.

In \lemref{lm:SO_inner_product_spaces1} below we prove that
\[
\dist^\infty_{\g,\euc}(\SO{\g,\euc},\SO{L^\star\h,\euc}) = \dist^\infty_{\g,\h}(L,\SO{\g,\h}).
\]
Together with the norm inequality \eqref{eq:equiv_of_norms}, we get
\[
\dist_{\g,\euc}(\SO{\g,\euc},\SO{L^\star\h,\euc}) \le k\, \dist_{\g,\h}(L,\SO{\g,\h}) =k\, \DIS L.
\]
In  \lemref{lm:SO_inner_product_spaces2} below we prove that
\[
\left| \dist_{\g,\euc}(A\circ L, \SO{L^\star \h,\euc}) - L^* \dist_{\h,\euc}(A, \SO{ \h,\euc}) \right| 
\le  (L^*|A|+k) \,\DIS L.
\]
Putting everything together,
\[
|W_{(\M,\g)}(A\circ L)^{1/p} - L^* W_{(\N,\h)}(A)^{1/p} | \le(L^*|A|+2k) \, \DIS L,
\]
which conclude the proof.
\end{proof}

\begin{lemma}
\label{lm:SO_inner_product_spaces1}
Let $(V,\g)$ and $(W,\h)$ be two oriented $k$-dimensional inner-product spaces, and let $L:V\to W$ be an isomorphism. Then, for any metric $\mathfrak{r}$ on $V$, 
\[
\dist^\infty_{\mathfrak{r},\euc}( \SO{\g,\euc}, \SO{L^\star \h, \euc} ) = \dist^\infty_{\mathfrak{r},\h}(L, \SO{\g,\h}).
\]
\end{lemma}

\begin{proof}
Let $R\in \SO{\g,\euc}$ and $Q\in \SO{L^\star \h,\euc}$; both are isomorphisms $V\to \R^k$. The (operator norm) distance between $R$ and $Q$ is
\[
\begin{split}
d_\infty^2(R,Q) & = 
\sup_{\|v\|_\mathfrak{r}=1} \|(R-Q)v\|_\euc^2 =
\sup_{\|v\|_\mathfrak{r}=1} \brk{\il Rv,Rv\ir_\euc + \il Qv,Qv\ir_\euc -2\il Rv,Qv\ir_\euc} \\
& = \sup_{\|v\|_\mathfrak{r}=1} \brk{\il v,v\ir_\g + \il v,v\ir_{L^\star \h} -2\il Q Q^{-1} Rv,Qv\ir_\euc} \\
& = \sup_{\|v\|_\mathfrak{r}=1} \brk{\|v\|^2_\g + \|L v\|^2_{\h} -2\il L Q^{-1} Rv,L v\ir_{\h}},
\end{split}
\]
where in the last step we used the fact that for every $u,v\in V$, 
\[
\il Qu,Qv \ir_\euc = \il u,v \ir_{L^\star\h} = \il Lu,Lv \ir_\h.
\]
Denote $S = L Q^{-1} R : V \to W$, and observe that $S\in \SO{\g,\h}$ as
\[
\begin{split}
\il Sv, Su\ir_{\h} & = \il L Q^{-1} Rv, L Q^{-1} Ru \ir_{\h} = \il Q^{-1}Rv , Q^{-1}Ru\ir_{L^\star\h} \\
	& = \il Rv , Ru\ir_{\euc} = \il v , u\ir_{\g}.
\end{split}
\]
Also,
\[
\begin{split}
d_\infty^2(L,S) & = \sup_{\|v\|_\mathfrak{r}=1} \|(L-S)v\|_{\h}^2 = 
\sup_{\|v\|_\mathfrak{r}=1} \brk{\|L v\|^2_{\h} + \| v\|^2_{\g} -2\il L v,Sv\ir_{\h}} 
 = d^2_\infty(R,Q).
\end{split}
\]
It follows that for every $R\in \SO{\g,\euc}$ and $Q\in \SO{L^\star \h,\euc}$,
\[
\dist^\infty_{\mathfrak{r},\euc}(R, Q ) \ge \dist^\infty_{\mathfrak{r},\h}( L, \SO{\g,\h}),
\]
which implies that
\[
\dist^\infty_{\mathfrak{r},\euc}( \SO{\g,\euc}, \SO{L^\star \h,\euc} ) \ge \dist^\infty_{\mathfrak{r},\h}( L, \SO{\g,\h}).
\]

For the reverse inequality, the same arguments imply that for every  $S\in  \SO{\g,\h}$ and $Q \in \SO{L^\star \h,\euc}$, $R = Q \,L^{-1} \,S\in \SO{\g,\euc}$ satisfies 
\[
d_\infty(R,Q) = d_\infty(L,S).
\]
Taking $S$ to be a minimizer for the right-hand side, we obtain that for every $Q\in \SO{L^\star \h,\euc}$,
\[
\dist^\infty_{\mathfrak{r},\euc}( \SO{\g,\euc}, Q ) \le \dist^\infty_{\mathfrak{r},\h}( L, \SO{\g,\h}).
\]
Similarly, since for every  $S\in  \SO{\g,\h}$ and $R \in  \SO{\g,\euc}$, $Q = R S^{-1} L \in \SO{L^\star \h,\euc}$ satisfies $d(R,Q) = d(L,S)$ we obtain that for every $R\in  \SO{\g,\euc}$,
\[
\dist^\infty_{\mathfrak{r},\euc}(R, \SO{L^\star \h,\euc} ) \le \dist^\infty_{\mathfrak{r},\h}( L, \SO{\g,\h}).
\]
From the definition of Hausdorff distance, these two inequalities imply that
\[
\dist^\infty_{\mathfrak{r},\euc}( \SO{\g,\euc}, \SO{L^\star \h,\euc} ) \le \dist^\infty_{\mathfrak{r},\h}(L, \SO{\g,\h}).
\] 
\end{proof}

\begin{lemma}
\label{lm:SO_inner_product_spaces2}
Let $(V,\g)$ and $(W,\h)$ be two oriented $k$-dimensional inner-product spaces, and let $L:V\to W$ be an isomorphism and $A:W\to \R^k$. Then
\[
\left| \dist_{\g,\euc}(A\circ L, \SO{L^\star \h,\euc}) - \dist_{\h,\euc}(A, \SO{ \h,\euc}) \right| \le  
(|A|+k) \, \DIS L.
\]
\end{lemma}

\begin{proof}
Let $B:W\to \R^k$. Then, for every $Q\in \SO{\g,\h}$,
\[
| \,|B|_{\h,\euc} - |B\circ L|_{\g,\euc} | = | \,|B\circ Q |_{\g,\euc} - |B\circ L|_{\g,\euc} | \le |B\circ (Q - L)|_{\g,\euc} \le |B|_{\h,\euc} \, |Q-L|_{\g,\h},
\]
where we used the sub-multiplicativity of the Frobenius norm. Hence,
\[
| \,|B|_{\h,\euc} - |B\circ L|_{\g,\euc} | \le |B|_{\h,\euc} \, \dist_{\g,\h}(L, \SO{\g,\h}).
\]
Take $B = A - R$ with $R\in \SO{ \h,\euc}$.
Then, $|B|_{\h,\euc} \le |A|_{\h,\euc} + k$, and
\[
\begin{split}
(|A|_{\h,\euc}+k) \, \dist_{\g,\h}(L, \SO{\g,\h}) 
&\ge |B|_{\h,\euc} \, \dist_{\g,\h}(L, \SO{\g,\h}) \\
&\ge |B|_{\h,\euc} - |B\circ L|_{\g,\euc} \\
&\ge \dist_{\h,\euc}(A, \SO{ \h,\euc}) - |B\circ L|_{\g,\euc} \\
&= \dist_{\h,\euc}(A, \SO{ \h,\euc}) - |A\circ L - R\circ L|_{\g,\euc}.
\end{split}
\] 
Since $R\circ L\in\SO{L^\star\h,\euc}$ and this holds for all $R\in \SO{ \h,\euc}$ we obtain
\[
\begin{split}
& (|A|_{\h,\euc}+k) \, \dist_{\g,\h}(L, \SO{\g,\h}) \\
	&\qquad \ge \dist_{\h,\euc}(A, \SO{ \h,\euc}) - \dist_{\g,\euc}(A\circ L, \SO{L^\star \h,\euc}).
\end{split}
\]
Repeating the same argument the other way around we obtain an absolute value in the second line.
\end{proof}

\section{$\Gamma$-convergence of elastic energies of converging manifolds}
\label{sec:Gamma_conv}

Let $\frakM$ be a class of $k$-dimensional body manifolds. Fix $p\in(1,\infty)$, and let $W$ be a $p$-regular energy density on $\frakM$, with $E$ the associated energy functional. For $(\M,d,\tM,\g)\in \frakM$, denote
\[
\Gamma E_{\M} = 
\begin{cases}
	\int_\M QW_{(\M,\g)}(du) \,\Volume & u\in W^{1,p}(\M;\R^k), \\
	+\infty & u\in L^p(\M;\R^k) \setminus W^{1,p}(\M;\R^k),
\end{cases}
\]
where $QW_{(\M,\g)}$ is the quasi-convex envelope of $W_{(\M,\g)}$ (see Section 3.4 in \cite{KM14} for a discussion on quasi-convexity in Riemannian settings).

In this section we prove $\Gamma$-convergent results for a sequence $E_{\M_n}$, where $\M_n\in \frakM$ is a convergent sequence of body manifolds. 
In \secref{sec:Gamma_conv2} we prove $\Gamma$-convergence, or establish an upper bound to $\Gamma$-convergent subsequences, depending on whether $\M_n\to\M$ uniformly or in the mean.
In \secref{sec:conv_of_minimizers} we adapt to our setting the standard convergence of minimizers for $\Gamma$-convergent (sub)sequences.

\subsection{$\Gamma$-convergence}
\label{sec:Gamma_conv2}

\begin{theorem}
\label{thm:Gamma_convergence_general}
Let $\M_n,\M \in \frakM$, then the following holds:
\begin{enumerate}
\item If $\M_n\to\M$ uniformly, then $E_{\M_n}$ $\Gamma$-converges to $\Gamma E_{\M}$.
\item If $\M_n\to\M$ in the mean, then the $\Gamma$-limit of every $\Gamma$-convergent subsequence of $E_{\M_n}$ is bounded from above by $\Gamma E_{\M}$.
\end{enumerate}
The $\Gamma$-convergence is with respect to the $L^p$-topology induced by some realization $F_n:\M\to\M_n$ of the convergence.
\end{theorem}

Note that although the topology depends on the choice of realizations $F_n$, neither the $\Gamma$-limit (in the first case) or the bound on the $\Gamma$-limit (in the second case) depends on this choice.

\begin{proof}
For succinctness, we will write $E_n = E_{\M_n}$, $E = E_{\M}$ and $\Gamma E = \Gamma E_{\M}$. Similarly we will write $W_n = W_{(\M_n,\g_n)}$ and $W_\infty = W_{(\M,\g)}$.

Let $E_\infty$ be the $\Gamma$-limit of a (not-relabeled) subsequence of $E_n$. Such a subsequence always exists by the general compactness theorem of $\Gamma$-convergence (see Theorem~8.5 in \cite{Dal93} for the classical result, or Theorem 4.7 in \cite{KS08} for the case where each functional is defined on a different space).

Part 2, which only assumes convergence in the mean, states that $E_\infty \le \Gamma E$.
This upper bound follows from Propositions \ref{pn:G_conv_L_p_infty_general} and \ref{pn:G_conv_L_p_upper_bound}.

To prove Part 1, which assumes uniform convergence, it is enough to prove that $E_\infty = \Gamma E$.
Indeed, since by the compactness theorem, every sequence has a $\Gamma$-converging subsequence, the Urysohn property of $\Gamma$-convergence (see Proposition 8.3 in \cite{Dal93}) implies that if all converging subsequences converge to the same limit, then the entire sequence converges to that limit.
\propref{pn:G_conv_L_p_lower_bound} establishes the lower bound $E_\infty \ge \Gamma E$, which together with the upper bound concludes the proof.
\end{proof}

\begin{proposition}[Infinity case]
\label{pn:G_conv_L_p_infty_general}
Assume $\M_n\to \M$ in the mean, and let $u\in L^p(\M;\R^k) \setminus W^{1,p}(\M;\R^k)$. Then $E_\infty[u] = \infty = \Gamma E[u]$.
\end{proposition}

\begin{proof}
Suppose, for contradiction, that $E_\infty[u]<\infty$. Let $u_n\to u$ be a recovery sequence, namely,
\[
\lim_{n\to\infty} E_n[u_n] =  E_\infty[u] < \infty.
\] 
W.l.o.g. we may assume that $E_n[u_n]<\infty$ for all $n$, and in particular,
$u_n\in W^{1,p}(\M_n,\R^k)$. 
The coercivity of $W_n$ implies that
\[
\sup_n \int_{\M_n}  |du_n|_{\g_n,\euc}^p \,\Volumen < \infty.
\]
Thus, $u_n$ is uniformly bounded in $W^{1,p}$, and by \lemref{lm:L_p_convergence},
$u_n\circ F_n$ weakly converges (modulo a subsequence) in $W^{1,p}(\M;\R^k)$. 
By the uniqueness of the limit, this limit is $u$, hence $u\in W^{1,p}(\M;\R^k)$, which is a contradiction. \end{proof}

\begin{proposition}[Upper bound]
\label{pn:G_conv_L_p_upper_bound}
Assume $\M_n\to\M$ in the mean.
Then, for every $u\in W^{1,p}(\M;\R^k)$,
\[
E_\infty[u] \le \Gamma E[u].
\]
\end{proposition}

\begin{proof}
Let $u\in W^{1,p}(\M;\R^k)$. 
Define $u_n = u\circ F_n^{-1} \in L^p(\M_n;\R^k)$. Trivially, $u_n \to u$ in $L^p$,  
and 
by the definition of $\Gamma$-limit,
\beq
\label{eq:upper_bound_0}
E_\infty[u] \le \liminf_n E_n[u_n]. \nonumber
\eeq
We now show that
\beq
\label{eq:upper_bound_1}
\lim_n E_n[u_n] = E[u].
\eeq
Since $|dF_n^{-1}|$ is uniformly bounded, $u_n \in W^{1,p}(\M_n;\R^k)$. Therefore, \eqref{eq:upper_bound_1} reads
\[
\lim_n \int_{\M_n} W_n(d(u\circ F_n^{-1})) \,\Volumen = \int_{\M} W_\infty(du) \,\Volume.
\]

First,
\beq
\label{eq:upper_bound_2}
\begin{split}
&\lim_n \int_{\M_n} W_n(d(u\circ F_n^{-1})) \,\Volumen = \lim_n \int_{\M} F_n^*W_n(d(u\circ F_n^{-1})) \,\VolumeFn \\
				&\qquad = \lim_n \int_{\M} F_n^*W_n(d(u\circ F_n^{-1})) \,\Volume +  \int_{\M} F_n^*W_n(d(u\circ F_n^{-1})) \brk{1-\frac{\Volume}{\VolumeFn}}\,\VolumeFn \\
				&\qquad = \lim_n \int_{\M} F_n^*W_n(d(u\circ F_n^{-1})) \,\Volume.
\end{split} 
\eeq
In the passage from the second to the third line we used the boundedness \eqref{eq:boundedness} of $W$  and the uniform convergence  \eqref{eq:volume_conv} of the volume forms.

Second, 
\beq
\label{eq:upper_bound_3}
\begin{split}
&\left| \brk{\int_{\M} F_n^*W_n(d(u\circ F_n^{-1})) \,\Volume}^{1/p}  - \brk{\int_{\M} W_\infty(du) \,\Volume }^{1/p} \right|\\
	&\qquad\le \brk{\int_{\M} \left| F_n^*W_n(d(u\circ F_n^{-1}))^{1/p} - W_\infty(du)^{1/p}  \right|^p \,\Volume }^{1/p}\\
	&\qquad \le C \brk{\int_\M (1 + |du|)^p \, \dist^p_{F_n^*\g_n,\g} (dF_n^{-1}, \SO{F_n^*\g_n,\g}) \, \Volume}^{1/p}.
\end{split}
\eeq
In the passage from the first to the second line we used the reverse triangle inequality, and in the passage to the third line we used 
the Lipschitz continuity \eqref{eq:property_d_energy_density} of $W$, with $L = d(F_n^{-1})$ and $A=du$.

Since by \eqref{eq:reverse_asymptotic_rigdity}, $\dist_{F_n^*\g_n,\g} (dF_n^{-1}, \SO{F_n^*\g_n,\g}) \to 0$ in $L^p$, we can assume by moving to a subsequence that this sequence converges almost everywhere.
Let $\e>0$. By Egorov's  theorem, there exists an $A\subset \M$ such that $\text{Vol}_\g(\M\setminus A) < \e$ and $\dist_{F_n^*\g_n,\g} (dF_n^{-1}, \SO{F_n^*\g_n,\g}) \to 0$ uniformly on $A$. 
Since $|du|$ is in $L^p$ and $\dist_{F_n^*\g_n,\g} (dF_n^{-1}, \SO{F_n^*\g_n,\g})$ is bounded uniformly by some constant $C'$, we obtain that
\beq
\label{eq:upper_bound_4}
\begin{split}
&\limsup_n \int_{\M} (1+|du|)^p \, \dist^p_{F_n^*\g_n,\g} (dF_n^{-1}, \SO{F_n^*\g_n,\g}) \,\Volume \\
	&\qquad \le \limsup_n \int_{\M\setminus A} (1+ |du|)^p \, \dist^p_{F_n^*\g_n,\g} (dF_n^{-1}, \SO{F_n^*\g_n,\g}) \,\Volume\\
	&\qquad \le \limsup_n C'\int_{\M\setminus A} (1+ |du|)^p \,\Volume.
\end{split}
\eeq
Since $\M\setminus A$ is arbitrary small and $|du|\in L^p$, the righthand side is arbitrary small, hence
\eqref{eq:upper_bound_3} and \eqref{eq:upper_bound_4} imply that
\beq
\label{eq:upper_bound_5}
\lim_n \int_{\M} F_n^*W_n(d(u\circ F_n^{-1})) \,\Volume  = \int_{\M} W_\infty(du) \,\Volume.
\eeq
Together with \eqref{eq:upper_bound_2}, \eqref{eq:upper_bound_1} follows.

We therefore obtain that for every $u\in W^{1,p}(\M;\R^k)$
\beq
\label{eq:upper_bound_200}
E_\infty[u] \le E[u].
\eeq
Together with \propref{pn:G_conv_L_p_infty_general}, we obtain that \eqref{eq:upper_bound_200} holds for every $u\in L^p(\M;\R^k)$. 
Since $E_\infty$ is a $\Gamma$-limit with respect to the $L^p$ topology, it is lower-semicontinuous (see Proposition 6.8 in \cite{Dal93} or Lemma 4.6 in \cite{KS08}), and
\beq
\label{eq:upper_bound_201}
E_\infty \le \tilde{E},
\eeq
where $\tilde{E}$ is the lower semicontinuous envelope of $E$ with respect to the strong $L^p$ topology. 
We complete the proof by showing that $\tilde{E} = \Gamma E$.
The argument is essentially the same as in the proof of Proposition 4.3 in \cite{KM14}, using Lemma 5 in \cite{LR95} and the results of \cite{AF84} (see Appendix B in \cite{KM14} for the relevant generalization of \cite{AF84} to manifolds).
\end{proof}

\begin{proposition}[Lower bound]
\label{pn:G_conv_L_p_lower_bound}
Assume $\M_n\to\M$ uniformly. Then, for every $u\in W^{1,p}(\M;\R^k)$,
\[
E_\infty[u] \ge \Gamma E[u] .
\]
\end{proposition}

\begin{proof}
Let $u\in W^{1,p}(\M;\R^k)$, and let $u_n\to u$ be a recovery sequence. 
If $E_\infty[u] = \infty$, then the claim is trivial. 
Otherwise, we may assume that $u_n\in W^{1,p}(\M_n;\R^k)$ for all $n$.
By the coercivity of $W_n$, $u_n$ is bounded in $W^{1,p}$, hence $u_n\circ F_n$ is bounded in $W^{1,p}(\M;\R^k)$ and weakly $W^{1,p}$-converges (modulo a subsequence) to $u$.

We will show that
\beq
\label{eq:lower_bound0}
\begin{split}
E_\infty[u] &= \lim_n E_n[u_n] = \lim_n \int_{\M_n} W_n(du_n) \,\Volumen \\
		&= \lim_n \int_{\M} W_\infty(d(u_n\circ F_n)) \,\Volume \\
		&\ge \lim_n \int_\M QW_\infty(d(u_n\circ F_n)) \,\Volume \\
		&\ge \int_\M QW_\infty(du) \,\Volume = \Gamma E[u].
\end{split}
\eeq
The passage from the second to the third line follows from the definition of the quasi-convex envelope. The passage from the third to the fourth line follows from the weak lower-semicontinuity of an integral functional with a Carath\'eodory quasiconvex integrand (see Section 3.4 in \cite{KM14} for details). The rest of the proof derives the equality between the first and the second line.

First, by the same arguments as in \eqref{eq:upper_bound_2},
\beq
\label{eq:lower_bound1}
\begin{split}
\lim_n \int_{\M_n} W_n(du_n) \,\Volumen = \lim_n \int_{\M} F_n^*W_n(du_n) \,\Volume .
\end{split} 
\eeq

Second, we use the Lipschitz continuity \eqref{eq:property_d_energy_density} of $W$, with $L = dF_n$ and $A=du_n$, and obtain that
\beq
\label{eq:lower_bound2}
\begin{split}
&\left| \brk{\int_{\M} F_n^*W_n(du_n) \,\Volume}^{1/p}  - \brk{\int_{\M} W_\infty(d(u_n\circ F_n)) \,\Volume }^{1/p}\right|\\
	&\qquad \le \brk{\int_{\M} \left|F_n^*W_n(du_n)^{1/p} - W_\infty(d(u_n\circ F_n))^{1/p}\right|^p \,\Volume }^{1/p}\\
	&\qquad \le C \brk{\int_\M (1 + F_n^*|du_n|)^p \, \dist_{\g,F_n^*\g_n}^p (dF_n, \SO{\g,F_n^*\g_n}) \, \Volume }^{1/p} \\
	&\qquad \le C \left\|\dist_{\g,F_n^*\g_n} (dF_n, \SO{\g,F_n^*\g_n}) \right\|_\infty \brk{\int_\M (1 + F_n^*|du_n|)^p \,  \, \Volume }^{1/p} \to 0
\end{split}
\eeq
using the uniform convergence \eqref{eq:uniform_asymptotic_rigidity} and the fact that $F_n^*|du_n|$ is uniformly bounded in $L^p(\M,\g)$.

Using \eqref{eq:lower_bound1} and \eqref{eq:lower_bound2} we obtain that
\beq
\label{eq:lower_bound4}
\lim_n E_n[u_n] = \lim_n \int_{\M_n} W_n(du_n) \,\Volumen= \lim_n \int_{\M} W_\infty(d(u_n\circ F_n)) \,\Volume
\eeq
which completes the proof.
\end{proof}

\subsection{Convergence of minimizers}
\label{sec:conv_of_minimizers}

The following proposition is a standard convergence of minimizers result that typically accompanies $\Gamma$-convergence results. 

\begin{proposition}
\label{pn:conv_of_min}
Assume that $\M_n\to\M$ either uniformly or in the mean, and that $E_{\M_n}$ $\Gamma$-converges to $E_\infty$ (in the case of uniform convergence we always have  $E_\infty=\Gamma E_{\M}$).
Let $u_n\in W^{1,p}(\M_n;\R^k)$ be a sequence of (approximate) minimizers of $E_{\M_n}$, and denote by $\overline{u_n}$ the mean of $u_n$. Then the translated sequence $u_n-\overline{u_n}$ is relatively compact (with respect to the $L^p$ topology defined by $F_n$), and all its limits points are minimizers of $E_\infty$.
Moreover,
\[
\lim_{n\to\infty} \,\,\,\inf_{L^p(\M_n;\R^k)} E_{\M_n} = \min_{L^p(\M;\R^k)} E_\infty.
\]
\end{proposition}

\begin{proof}
Once again, we write $E_n = E_{\M_n}$ and $\Gamma E = \Gamma E_{\M}$.
Let $u_n$ be a sequence of approximate minimizers of $E_n$. Since $E_n$ is invariant to translations (in the sense that $E_n[u_n] = E_n[u_n + x]$ for every $x\in \R^k$), we will assume w.l.o.g. that $\overline{u_n}=0$. 
We first prove that it is relatively compact, i.e. that every subsequence (not relabeled) of $u_n$ has a subsequence  converging in $L^p$.

Let $w\in W^{1,p}(\M_n;\R^k)$ be arbitrary and let $w_n\in L^p(\M_n;\R^k)$ be a recovery sequence for $w$. Then, by Theorem \ref{thm:Gamma_convergence_general},
\[
\inf_{L^p} \,\,E_n[\cdot] \le E_n[w_n] \limarrow_{n\to\infty} E_\infty[w] \le \Gamma E[w]  <\infty.
\]
This shows  that $\inf_{L^p} E_n[\cdot]$ is a bounded sequence.

It follows that $E_n[u_n]$ is bounded, hence, by coercivity, $du_n$ is uniformly bounded in $L^p$. Together with the Poincar\'e inequality, we obtain that $u_n\circ F_n$ is uniformly bounded in $W^{1,p}(\M;\R^k)$. This implies the existence of a (not relabeled) subsequence $u_n\to u$ in $L^p$, proving the relative compactness of $u_n$.

We now prove that $u$ is a minimizer of $E_\infty$. Let $w\in L^p(\M;\R^k)$ be an arbitrary function, and  let $w_n\in L^p(\M_n;\R^k)$ be a recovery sequence for $w$. Then,
\[
E_\infty[w] = \lim_{n\to\infty} E_n[w_n] \ge \lim_{n\to\infty} \inf_{L^p} E_n[\cdot] = \lim_{n\to\infty} E_n[u_n] \ge E_\infty[u],
\]
where the last inequality follows from the lower-semicontinuity property of $\Gamma$-limits.
Since $w$ is arbitrary, $u$ is a minimizer of $E_\infty$.
Moreover, by choosing $w=u$ we conclude that
\[
E_\infty[u]  = \lim_{n\to\infty} \inf_{L^p} E_n[\cdot].
\]
\end{proof}

\section{Discussion}
\label{sec:discussion}

In this paper we proved a $\Gamma$-convergence result for elastic models of uniformly converging manifolds.
This result is intrinsic, in the following senses: first, it does not depend on the parametrizations of the manifolds $\M_n$ and $\M$. 
Second, while the $L^p$-topology described in Section \ref{sec:L_p_top} depends on the maps $F_n$, the limiting functional $\Gamma E_\M$ itself is independent of these maps.
That is, $\Gamma E_\M[u]$ is defined independently of the choice of maps, even though recovery sequences converging to $u$ depend on them.
The intrinsic nature of the limit model highlights the geometric nature of non-Euclidean elasticity, which is sometimes obscured by choices of coordinates and maps.

For manifolds that converge in the mean, we do not obtain a $\Gamma$-convergence result, but, similarly, the $\Gamma$-upper bound is independent of parametrization and of the maps $F_n$.

\paragraph{Boundary conditions and external forces}
For the sake of clarity, we limited our analysis to unconstrained systems, i.e., systems without external forces and without boundary constraints.
Forces and boundary conditions can be included in a standard way. Note, however, that boundary conditions should be specified for configurations of each of the manifolds $\M_n$ and $\M$, and the maps $F_n:\M\to \M_n$ that realize the convergence should map admissible configurations to admissible configurations.
That is, the maps $F_n$ must satisfy the condition that $u_n\in W^{1,p}(\M_n; \R^k)$ is $\M_n$-admissible if and only if $u_n\circ F_n$ is $\M$-admissible.

\paragraph{Applications of the main theorem to dislocation theory}
As discussed in the Introduction, bodies with continuously-distributed dislocations are commonly modeled as smooth Weitzenb\"ock manifolds $(\M,\g,\nabla)$, where the affine connection $\nabla$ is metrically consistent with $\g$ and flat (hence, uniquely determined by its torsion tensor).
In the terminology of this paper, this corresponds to a body manifold $(\M,d,\M,\g,\nabla)$ (since $\M$ is smooth $\tM=\M$).
This model is, naturally, viewed as a limit of bodies with finitely many dislocations $(\M_n,d_n,\tM_n,\g_n,\nabla_n)$, were $\nabla_n$ is the Levi-Civita connection, i.e., torsionless.

In this paper we associate with each body manifold $(\M,d,\tM,\g)$ an elastic energy functional $E_{(\M,\g)}$. 
In the case of continuously-distributed dislocations, we would expect to have an energy functional that depends on the connection, namely,
 $E_{(\M,\g,\nabla)}$.
When there are no distributed dislocations, $\nabla$ is the Levi-Civita connection, $\nabla^{LC}$, hence it is natural to assume $E_{(\M,\g,\nabla^{LC})}=E_{(\M,\g)}$, so that $E_{(\M,\g,\nabla)}$ extends $E_{(\M,\g)}$.
Since a body with a continuous distribution of dislocations $(\M,d,\M,\g,\nabla)$ is an effective model for bodies with finitely many dislocations $(\M_n,d_n,\tM_n,\g_n,\nabla_n^{LC})$, we expect $E_{(\M,\g,\nabla)}$ to be a limit of the elastic energies $E_{(\M_n,\g_n)}$ (up to relaxation).

Two questions arise in this context: 
First, is $E_{(\M,\g,\nabla)}$ well-defined as a limit of energies $E_{(\M_n,\g_n)}$, independently of the converging sequence of manifolds?
Second, how does $E_{(\M,\g,\nabla)}$ depend on $\nabla$?

The first part of \thmref{thm:Gamma_convergence_general} implies that the limiting elastic energy does not depend on the limiting process as long as the sequence of manifolds with finitely-many dislocations converges uniformly (like the variation of the constructions in \cite{KM15,KM15b} presented in Example~2 in Section \ref{sec:Remarks_and_Examples}). In this case, the limiting energy does \emph{not} depend on the connection $\nabla$ (or equivalently on the torsion). 
In other words, the limiting elastic model \emph{is only sensitive to the metric structure of the limit manifold}.

If one rather considers a larger class of body manifolds that converges to the limit $(\M,d,\M,\g,\nabla)$, including the original constructions in \cite{KM15,KM15b}  (which converges only in the mean, see Example~1 in Section \ref{sec:Remarks_and_Examples}), then our results do not guarantee the existence of a $\Gamma$-limit  independent of the converging sequence. 
However, one would still expect that if a specific sequence of bodies with dislocations has a $\Gamma$-limit energy, then the effect of the torsion would be an additional compatibility constrain, and hence would increase the energy compared to the torsion-free case. 
In other words, the inequality $E_{(\M,\g,\nabla)}\ge E_{(\M,\g)}$ is expected.
The second part of Theorem~\ref{thm:Gamma_convergence_general} shows that it is not the case, as the limit energy is bounded from \emph{above} by that determined by the metric.
In particular, if $(\M,\g)$ can be isometrically embedded in $\R^k$ (e.g.~as in the main example of \cite{KM15}), then $(\M,d,\M,\g,\nabla)$ has a zero energy embedding in $\R^k$, regardless of $\nabla$ and the converging sequence.

\paragraph{Other applications of the main theorem}
The first part of Theorem \ref{thm:Gamma_convergence_general} also holds for approximations of a surface $\M$ by Euclidean triangles $\M_n$, as described in Example~3 in Section~\ref{sec:Remarks_and_Examples}.
The $\Gamma$-convergence still holds if one considers only maps $\M_n\to \R^k$ which are affine on every triangle.
Indeed, the only change in the proof  is to replace the recovery sequence $u_n:\M_n\to \R^k$ in Proposition \ref{pn:G_conv_L_p_upper_bound} with a piecewise affine sequence that $L^p$-converges to the same limit $u:\M\to\R^k$; this is always possible.
This implies the consistency of finite element approximations based on triangulations of surfaces (or on simplices in higher dimensions) and their piecewise affine embeddings into  Euclidean space.

Finally, our results establish the structural stability of elastic models that are $p$-regular according to \defref{def:elastic_energy}: if two metrics are arbitrarily close to each other (with respect to the sup-norm), then their elastic energies are arbitrarily close.
This observation validates experimental estimates of reference metrics via interpolations based on finite sets of  measured distances (e.g., \cite{SRS07}).

\paragraph{Other rigidity criteria}
The distortion of a linear map $\DIS A$ defined in \eqref{eq:def_local_dis} plays a role both in the definition of convergence of body manifolds in \secref{sec:conv_of_manifolds} and in the Lipschitz continuity property of $p$-regular energy densities in \defref{def:elastic_energy}.

In principle, one can choose other measures for the distortion of a linear map, for which other energy densities may be $p$-regular according to \defref{def:elastic_energy}.
If we change the definition of $\DIS A$ accordingly also in the definitions of converging body manifolds (Definitions \ref{def:distortions}--\ref{def:weak_convergence}), then our results do not change, providing that for uniformly converging body manifolds, the uniform convergence of the new distortion criterion in \defref{def:uniform_convergence} continues to imply uniform bi-Lipschitzness \eqref{eq:uni_bi_Lipschitz} and uniform volume convergence \eqref{eq:volume_conv}.  
Even if this is not true for the new distortion criterion, \eqref{eq:uni_bi_Lipschitz}  and \eqref{eq:volume_conv} can be assumed in addition to \eqref{eq:uniform_asymptotic_rigidity} in the definition of uniform convergence \ref{def:uniform_convergence} (as in the definition of mean convergence \ref{def:weak_convergence}), and the proof will still hold.

\paragraph{Open questions}
Outside the context of dimension reduction, this is, to the best of our knowledge, the first paper to consider $\Gamma$-convergence of elastic energies of converging manifolds. Unlike dimension reduction, where the converging manifolds are ordered by an inclusion relation, here the notion of convergence allows for varying topologies and metric structures.
Naturally, there remain numerous open questions, among which are:

\begin{enumerate}
\item Do the elastic energies $\Gamma$-converge in the case of manifolds converging in the mean?
Even if such a result does not hold in general, it is of interest to determine whether it holds for the specific sequence of manifolds with dislocations considered in \cite{KM15,KM15b} 
 (see also Example~1 in Section~\ref{sec:Remarks_and_Examples}).

\item It would also be interesting to relax some of the assumptions on the elastic energy densities. In particular, for a more physical model one may want to  modify the growth condition in Definition~\ref{def:elastic_energy}, such to include densities tending to infinity when deformations tend to be singular (the relaxation of the growth condition is of interest in many other contexts as well, see \cite{CD15} for details).

\item This paper considers ``bulk" elasticity---the embedding of a $k$-dimensional manifold in the $k$-dimensional Euclidean space. Another main theme in elasticity theory (both classical and non-Euclidean) is the derivation of dimensionally reduced models for bodies with one or more slender dimensions (see e.g. \cite{LR95,FJM02b,KS14,KM14}).
An interesting question concerns the two-parameter limit of changing metrics and dimension reduction. A result in this direction would also relate to the von-K\'arm\'an limits of slender bodies whose metrics tend to a Euclidean metric; such a situation was treated in \cite{LMP10}.

\item Another question, which unlike the previous ones is of geometric nature rather than analytic,  concerns the role of global distortion in manifolds that converge in the mean. An asymptotically vanishing global distortion is part of our mean convergence definition (Condition~2 in \defref{def:weak_convergence}). Its only role in the present paper is to guarantee the uniqueness of the limit (as vanishing distortion implies Gromov-Hausdorff convergence); it doesn't play any role in the subsequent analysis.

It would be interesting to understand whether this condition can be omitted, that is to say, whether asymptotic vanishing of the mean local distortions (Condition~3 in \defref{def:weak_convergence}) suffices to define a notion of convergence (i.e., that the limit does not depend on the choice of morphisms). 
Such a result would require rigidity estimates for Riemannian manifolds, analogous to Reshetnyak's generalization of Liouville's rigidity theorem (which is in a Euclidean setting, see \cite{Res67} for the original paper and \cite{FJM02b} for a modern restatement).
\end{enumerate}

\paragraph{Acknowledgements}
This research was supported by the Israel-US Binational Foundation (Grant No. 2010129), by the Israel Science Foundation (Grant No. 661/13) and by a grant from the Ministry of Science, Technology and Space, Israel and the Russian Foundation for Basic Research, the Russian Federation.

\footnotesize{
\bibliographystyle{amsalpha}
\newcommand{\etalchar}[1]{$^{#1}$}
\providecommand{\bysame}{\leavevmode\hbox to3em{\hrulefill}\thinspace}
\providecommand{\MR}{\relax\ifhmode\unskip\space\fi MR }
\providecommand{\MRhref}[2]{%
  \href{http://www.ams.org/mathscinet-getitem?mr=#1}{#2}
}
\providecommand{\href}[2]{#2}

}

\end{document}